\documentclass[reqno]{amsart}
\usepackage{CJK}
\usepackage{hyperref}
\usepackage{cite}
\usepackage{amsmath}
\usepackage{mathrsfs}

\numberwithin{equation}{section}

\newtheorem{theorem}{Theorem}[section]
\newtheorem{lemma}[theorem]{Lemma}
\newtheorem{definition}[theorem]{Definition}

\newtheorem{remark}[theorem]{Remark}

\allowdisplaybreaks

\newcommand{ \mint }{ {\int\hspace{-0.38cm}-}}

\begin{document}
	
	\title[\hfil Mixed local and nonlocal double phase parabolic equations] {Regularity of weak solutions for mixed local and nonlocal double phase parabolic equations}
	
	\author[B. Shang and C. Zhang  \hfil \hfilneg]
	{Bin Shang  and Chao Zhang$^*$}
	
	\thanks{$^*$Corresponding author.}

	\address{Bin Shang \hfill\break
		School of Mathematics, Harbin Institute of Technology,
		Harbin 150001, P.R. China} \email{shangbin0521@163.com}

	\address{Chao Zhang\hfill\break
		School of Mathematics and Institute for Advanced Study in Mathematics, Harbin Institute of Technology,
		Harbin 150001, P.R. China} \email{czhangmath@hit.edu.cn}

	.

	\maketitle

\begin{abstract}
We study the mixed local and nonlocal double phase parabolic equation
\begin{align*}
\partial_t u(x,t)-\mathrm{div}(a(x,t)|\nabla u|^{q-2}\nabla u) +\mathcal{L}u(x,t)=0
\end{align*}
in $Q_T=\Omega\times(0,T)$, where $\mathcal{L}$ is the nonlocal $p$-Laplace type operator. The local boundedness of weak solutions is proved by means of the De Giorgi-Nash-Moser iteration with the nonnegative coefficient function $a(x,t)$ being bounded. In addition, when $a(x,t)$ is H\"{o}lder continuous, we discuss the pointwise behavior and lower semicontinuity of weak supersolutions based on energy estimates and De Giorgi type lemma. Analogously, the corresponding results are also valid for weak subsolutions.
\end{abstract}

\section{Introduction}
\par
	
In this paper, we investigate some regularity results of the following problem
\begin{align}
\label{1.1}
\partial_t u(x,t)-\mathrm{div}(a(x,t)|\nabla u|^{q-2}\nabla u) +\mathcal{L}u(x,t)=0 \text{ \ \ in }Q_T,
\end{align}
where  $Q_T:=\Omega\times(0,T)$ with $T>0$ and $\Omega$ being a bounded domain in $\mathbb{R}^N$.
The fractional $p$-Laplace operator $\mathcal{L}$ is given by
\begin{align*}
\mathcal{L}u(x,t)=\mathrm{P.V.}\int_{\mathbb{R}^N}K(x,y,t)|u(x,t)-u(y,t)|^{p-2}\left(u(x,t)-u(y,t)\right)\,dy,
\end{align*}
where P.V. stands for the Cauchy principle value and $K(x, y, t)$ is the symmetric Kernel function satisfies
\begin{align*}
\frac{\Lambda^{-1}}{|x-y|^{N+sp}}\leq K(x,y,t)=K(y,x,t)\leq \frac{\Lambda}{|x-y|^{N+sp}} \ \ \ \forall t\in(0,T), \ \forall x,y\in\mathbb{R}^N
\end{align*}
for some $\Lambda\geq1$. Moreover, we suppose that 
\begin{align}
\label{1.2}
p,q\in (1,\infty), \quad  sp<q, \quad s\in (0,1)
\end{align}
and the coefficient function $a: Q_T\rightarrow \mathbb{R}^+$ satisfies
\begin{align}
\label{1.3}
0\leq a(x,t)\leq \|a\|_{L^\infty(Q_T)}.
\end{align}
\subsection{Overview of related literature}	
Eq. \eqref{1.1} we study in this paper is characterized by different type of operators depending on the size of the coefficient function $a(x,t)$. Observe that $\eqref{1.1}$ behaves as a purely nonlocal $p$-Laplace equation on $\{a(x,t)=0\}$ while leading to the problem involving different orders of differentiation on $\{a(x,t)>0\}$. For the study on regularity theory of nonlocal equations, we refer the readers to \cite{APG22,S19,DKP16,BLS21,L22,APT22,KW23,KKL19,DZZ21,C17}. The mixed local and nonlocal problem have been widely studied recently owing to its essential applications in plasma physics \cite{DLV21} and biology \cite{Bd13}. In the following, we will focus on introducing some known results for the mixed problems.

For the basic prototype of the mixed local and nonlocal equation
\begin{equation}
\label{1.4}
-\Delta u+(-\Delta)^su=0,
\end{equation}
the Harnack inequality and boundary Harnack principle were derived in \cite{F09} and \cite{CKSV12}, respectively; see also \cite{CK10} for the Harnack estimate of globally nonnegative solutions regarding the parabolic version of \eqref{1.4}. Moreover, a weak Harnack inequality with a tail term to the parabolic setting of \eqref{1.4} can be found in \cite{GK21}. Recently, Biagi-Dipierro-Valdinoci-Vecchi \cite{BDVV21, BDVV211} developed systematic research on mixed problems and established several results involving radial symmetry, regularity, maximum principle, and other qualitative properties of weak solutions to \eqref{1.4}.
For the nonlinear scenario of \eqref{1.4}
\begin{align}
\label{1.5}
-\Delta_p u+(-\Delta)^s_p u=0,
\end{align}
the regularity theory involving Harnack inequality, local boundedness, and H\"{o}lder continuity along with semicontinuity representative of weak solutions were discussed by Garain-Kinnunen \cite{GK}. When it comes to the parabolic version of \eqref{1.5}, Shang-Zhang studied the local H\"{o}lder regularity \cite{SZ22} and Harnack inequality \cite{SZ23} via analytic approach. For what concerns existence, uniqueness, local boundedness and higher H\"{o}lder regularity for weak solutions to \eqref{1.5} with the nonhomogeneous term were provided in \cite{GL23}. In particular, De Filippis-Mingione \cite{DM22} first proposed the functionals with nonstandard growth
\begin{align}
\label{1.6}
u \mapsto \int_{\Omega}\left(|D u|^q-f u\right)\,dx+\int_{\mathbb{R}^N} \int_{\mathbb{R}^N} \frac{|u(x)-u(y)|^p}{|x-y|^{N+sp}}\,dxdy,
\end{align}
and proved the maximal regularity of minimizers to \eqref{1.6} with $q>sp$. The case $q<sp$ can be seen as a natural extension of the above problem and the local behaviour including the boundedness, H\"{o}lder continuity and Harnack inequality of the minimizers of \eqref{1.6} were achieved by Ding-Fang-Zhang \cite{DFZ23}. Besides, Byun-Lee-Song \cite{BLS23} considered the regularity results for minimizers of functionals modeled after
\begin{align}
\label{1.7}
u \mapsto \int_{\mathbb{R}^N} \int_{\mathbb{R}^N} \frac{|u(x)-u(y)|^p}{|x-y|^{N+sp}}\,dxdy+\int_{\Omega} a(x)|D u|^q\,dx,
\end{align}
where $0\leq a(x)\leq\|a\|_{L^\infty(\Omega)}$ and $1<p\leq q$. Functional \eqref{1.7} is linked to the local or nonlocal double phase elliptic problems. We refer the readers to \cite{CM15,CM16,FM20,BCM15,FZ23,DP19} and references therein for the regularity results. On the other hand, for the parabolic double phase problem, the gradient higher integrability, Lipschitz truncation method and the Calder\'{o}n-Zygmund type estimates have been provided in \cite{K23,KKM22,KKS23}. Prasad-Tewary \cite{PT23} proved the local boundedness of variational solutions for the nonlocal double phase parabolic equation by variational method.

As far as we know, there are no results concerning the mixed local and nonlocal double phase parabolic equation \eqref{1.1} for the case $p\neq q$, even the coefficient function $a(x,t)$ is a constant. To this end, we are interested in proving the local boundedness by energy estimates and De Giorgi-Nash-Moser iteration. We would like to mention that our approach to establish energy estimates differs from \cite{PT23} since they discussed the variational solutions that can be proved as parabolic minimizers. Furthermore, based on energy estimates and De Giorgi type lemma, we take into account the lower semicontinuity and pointwise behavior of weak supersolutions to \eqref{1.1}, and the related results also hold true for subsolutions.


\subsection{Notations and Definitions}

For the sake of simplicity, we give some notations that will be used later. 

For radius $\rho>0$, let $B_\rho(x)$ be the domain centered at $x\in\mathbb{R}^N$. We denote the forward and backward space-time cylinders by
\begin{align*}
Q_{\rho,\theta}^+(x_0,t_0)=B_\rho(x_0)\times(t_0,t_0+\theta) \ \ \text{and} \ \ Q_{\rho,\theta}^-(x_0,t_0)=B_\rho(x_0)\times(t_0-\theta,t_0).
\end{align*}
Denote the centered cylinders as
\begin{align*}
Q_{\rho,\theta}(x_0,t_0)=Q_{\rho,\theta}^-(x_0,t_0)\cup Q_{\rho,\theta}^+(x_0,t_0).
\end{align*}
If not important, we will omit $(x_0,t_0)$ in the above symbols for clearly.

For any $a,b\in\mathbb{R}$, define
\begin{align*}
J_{p}(a, b)=|a-b|^{p-2}(a-b).
\end{align*}
Set
\begin{align*}
d\mu=d\mu(x,y,t)=K(x,y,t)\,dxdy,
\end{align*}
and $$data:=data(N,p,s,q,\Lambda).$$

We introduce the tail space as
\begin{align*}
L_\alpha^p(\mathbb{R}^{N}):=\left\{v \in L_{\rm{loc}}^p(\mathbb{R}^{N}): \int_{\mathbb{R}^N} \frac{|v(x)|^p}{1+|x|^{N+\alpha}}\,dx<+\infty\right\}, \quad p>0 \text{ and } \alpha>0.
\end{align*}
The tail appearing in estimates is denoted by
\begin{align}
\label{1.8}
\mathrm{Tail}_\infty(v;x_{0},\rho,I) &=\mathrm{Tail}_{\infty}(v; x_0,\rho,t_0-\theta,t_0)\nonumber\\
&:=\operatorname*{ess\sup}_{t \in I}\left(\rho^{sp} \int_{\mathbb{R}^{N} \backslash B_\rho(x_{0})} \frac{|v(x, t)|^{p-1}}{|x-x_0|^{N+sp}}\,dx\right)^{\frac{1}{p-1}}
\end{align}
with $(x_0, t_0)\in\mathbb{R}^N\times(0,T)$ and the interval $I=\left(t_0-\theta, t_0\right) \subseteq(0,T)$. By the definition, it is not hard to verify that $\mathrm{Tail}_\infty(v;x_0,\rho,I)<\infty$ for any $v \in L^{\infty}(I; L_{sp}^{p-1}(\mathbb{R}^{N}))$.

\smallskip

Now, we state the definition of weak solutions to \eqref{1.1} as follows.

\begin{definition}
\label{def-1-1}
We say that 
$$u\in L^q_{\rm{loc}}(0,T;W_{\rm {loc}}^{1,q}(\Omega)) \cap C_{\rm{loc}}(0,T;L_{\rm {loc}}^2(\Omega)) \cap L^{\infty}(0,T;L_{sp}^{p-1}(\mathbb{R}^N))$$
is a local weak sub(super-)solution to \eqref{1.1}, if for every closed interval $[t_1, t_2] \subseteq(0,T)$ and any compact set $K\subset \Omega$, integral equality
\begin{align}
\label{1.9}
&\int_K u(x,t_2)\varphi(x,t_2)\,dx-\int_K u(x,t_1)\varphi(x,t_1)\,dx
-\int_{t_1}^{t_2} \int_K u(x,t)\partial_t \varphi(x,t)\,dxdt\nonumber\\
&\quad+\int_{t_1}^{t_2}\int_K a(x,t)|\nabla u|^{q-2}\nabla u\cdot \nabla\varphi\,dxdt +\int_{t_1}^{t_2} \mathcal{E}(u,\varphi,t)\,dt \leq(\geq) 0
\end{align}
holds for every nonnegative test function $\varphi \in L^q_{\rm{loc}} (0,T;W_0^{1,q}(K))\cap W_{\rm{loc}}^{1,2}(0,T;L^2(K))$, where
\begin{align*}
\mathcal{E}(u,\varphi,t):=\frac{1}{2} \int_{\mathbb{R}^N} \int_{\mathbb{R}^N}&\Big[|u(x,t)-u(y,t)|^{p-2}(u(x,t)-u(y,t))\\
&\cdot (\varphi(x,t)-\varphi(y,t))K(x,y,t)\Big]\,dxdy.
\end{align*}
A function $u$ called a weak solution to \eqref{1.1} that is both a weak supersolution and a weak subsolution.
\end{definition}

Let $u$ be a locally essentially bounded below measure function in $Q_T$, the lower semicontinuous regularization of $u$ is given by
\begin{align*}
u_*(x,t):=\operatorname*{ess\lim\inf}_{(y,\hat{t}) \rightarrow(x,t)}u(y,\hat{t})=\lim_{\rho \rightarrow 0} \operatorname*{ess\inf}_{Q_{\rho,\theta}(x,t)} u  \text{ \ \ for every
}(x,t)\in Q_T.
\end{align*}
Likewise, for the measurable function $u$ that is essentially bounded above in $Q_T$, we define the upper semicontinuous regularization of $u$ as
\begin{align*}
u^*(x,t):=\operatorname*{ess\lim\sup}_{(y,\hat{t}) \rightarrow(x,t)}u(y,\hat{t})=\lim_{\rho \rightarrow 0} \operatorname*{ess\sup}_{Q_{\rho,\theta}(x,t)} u  \text{ \ \ for every
}(x,t)\in Q_T.
\end{align*}

Let $u\in L_{\rm{loc}}^1(Q_T)$, we denote
\begin{align*}
\mathcal{F}:=\left\{(x,t)\in Q_T:|u(x,t)|<\infty, \lim_{\rho \rightarrow 0} \mint_{Q_{\rho,\theta}(x,t)}|u(x,t)-u(y,\hat{t})|\,dyd \hat{t}=0\right\},
\end{align*}
then it follows from the Lebesgue differentiation theorem that $|\mathcal{F}|=|Q_T|$.

\subsection{Main results}

Our results can be stated as follows. The first one we present is the local boundedness result in the supercritical case. Denote
\begin{align*}
a \vee b:=\max \{a,b\}, \quad a_+:=\max \{a,0\}, \quad a_-:=-\min \{a,0\}.
\end{align*}

\begin{theorem} [Local boundedness]
\label{thm-1-2}
Suppose that \eqref{1.2} and \eqref{1.3} hold. Let $p> 2N/(N+2s)$, $\kappa:=1+2s/N$ and $q<p\kappa$. Assume that $u$ is a local weak subsolution to \eqref{1.1}. Let $(x_0,t_0) \in Q_T$, $R>0$ and $Q_R^- \equiv B_R(x_0) \times(t_0-R^{sp},t_0)$ such that $\bar{B}_R(x_0)\subseteq\Omega$ and $[t_0-R^{sp}, t_0] \subseteq(0,T)$. Then for any $\delta=\max\{2,p,q\}$, we have
\begin{align*}
\operatorname*{ess\sup}_{Q_{R/2}^-} u\leq \mathrm{Tail}_\infty\left(u_+;x_0,R/2,t_0-R^{sp}, t_0\right)+C(2+R^{sp-q})^{\frac{1}{\tau}}\left(\mint_{Q_R^-} u_+^\delta\,dx dt\right)^{\frac{s\delta}{\tau(N\kappa+s\delta)}} \vee 1,
\end{align*}
where $\tau=\min\{\delta-q,\delta-p,\delta-2\}$, and the constant $C>0$ only depends on $\|a\|_{L^\infty(Q_T)}$ and the data.
\end{theorem}

In the subcritical range $1<p\leq 2N/(N+2s)$, supposing that the weak solutions 
can be constructed as: for $r>\max\left\{2,\frac{N(2-p)}{sp}\right\}$, there exists a sequence $\{u_k\}_{k\in \mathbb{N}}$ of bounded weak solutions to \eqref{1.1} such that
\begin{align}
\label{1.10}
\|u_k\|_{L_{\rm{loc}}^\infty(0,T;L_{sp}^{p-1}(\mathbb{R}^N))}\leq C
\end{align}
and
\begin{align}
\label{1.11}
u_k\rightarrow u \quad \text{in } L_{\rm{loc}}^r(Q_T) \text{ as } k\rightarrow\infty.
\end{align}

\begin{theorem} [Local boundedness]
\label{thm-1-3}
Suppose that \eqref{1.2} and \eqref{1.3} hold. Let $1<p\leq 2N/(N+2s)$, $ \kappa=1+2s/N$, $q<p\kappa$ and $r>\max\left\{2,\frac{N(2-p)}{sp}\right\}$.
Assume that $u\in  L^r_{\rm loc}(Q_T)$ fulfilling \eqref{1.10} and \eqref{1.11} is a weak subsolution to \eqref{1.1}. Let $(x_0,t_0) \in Q_T$, and $Q_R^- \equiv B_R(x_0) \times(t_0-R^{sp},t_0)$ such that $\bar{B}_R(x_0)\subseteq\Omega$ and $[t_0-R^{sp}, t_0] \subseteq(0,T)$. Then we have
\begin{align*}
\operatorname*{ess\sup}_{Q_{R/2}^-} u\leq& \mathrm{Tail}_\infty\left(u_+;x_0,R/2,t_0-R^{sp}, t_0\right)\\
&+C\left(R^{sp-q}+2\right)^{\frac{1}{r-2-\beta}}\left(\mint_{Q_R^-} u_+^r\,dx dt\right)^{\frac{sp}{(N+sp)(r-2-\beta)}}\\
&\quad\vee\left(R^{sp-q}+2\right)^{\frac{1}{r-p-\beta}} \left(\mint_{Q_R^-} u_+^r\,dx dt\right)^{\frac{sp}{(N+sp)(r-p-\beta)}}\\
&\quad\vee\left(R^{sp-q}+2\right)^{\frac{1}{r-q-\beta}} \left(\mint_{Q_R^-} u_+^r\,dx dt\right)^{\frac{sp}{(N+sp)(r-q-\beta)}},
\end{align*}
where $\beta=N(r-p\kappa)/(N+sp)$, and $C>0$ only depends on  $\|a\|_{L^\infty(Q_T)},r$ and the data.
\end{theorem}

The following results are regarding the semicontinuous properties of weak solutions. Here, we  assume in addition that
\begin{align}
\label{1.12}
 0\leq a\in C^{\alpha,\frac{\alpha}{2}}(Q_T) \text{\ \ and \ \ }  q\leq sp+\alpha  \text{\ \ for some \ } \alpha\in(0,1].
\end{align}
The function $a\in C^{\alpha,\frac{\alpha}{2}}(Q_T)$ means
\begin{align*}
|a(x,t)-a(y,t)|\leq [a]_{\alpha,\frac{\alpha}{2}}|x-y|^\alpha \text{\ \ and \ \ } |a(x,t)-a(x,s)|\leq [a]_{\alpha,\frac{\alpha}{2}}|t-s|^{\frac{\alpha}{2}}
\end{align*}
for every $(x,y)\in\Omega$ and $(t,s)\in(0,T)$.
\begin{theorem}
\label{thm-1-4}
Let the assumptions \eqref{1.2} and \eqref{1.12} hold. Suppose that $u$ is a local weak solution to \eqref{1.1} that is bounded from below in $\mathbb{R}^N\times(0,T)$. Let $\mu^-$ be a number such that
\begin{align*}
\mu^-\leq \operatorname*{ess\inf}_{\mathbb{R}^N\times(0,T)} u.
\end{align*}
Then we have $u_*(x,t)=u(x,t)$ for all $(x,t)\in\mathcal{F}$. In particular, $u_*$ is a lower semicontinuous representative of $u$ in $Q_T$.
\end{theorem}

\begin{remark}
\label{rem-1-5}
Based on the same method of proving Theorem \ref{thm-1-4}, we can obtain that a subsolution of \eqref{1.1} has an upper semicontinuous representative in $Q_T$.
\end{remark}

We now present the pointwise behavior of weak supersolutions to \eqref{1.1}, which shows that the value of lower semicontinuous representative $u_*$ at some point can be recovered pointwise from the values at previous times.

\begin{theorem}
\label{thm-1-6}
Let the assumptions \eqref{1.2} and \eqref{1.12} hold. Suppose that $u$ is a local weak supersolution to \eqref{1.1} that is bounded from below in $\mathbb{R}^N\times(0,T)$.  Let $\mu^-$ be a number such that
\begin{align*}
\mu^-\leq \operatorname*{ess\inf}_{\mathbb{R}^N\times(0,T)} u.
\end{align*}
For a positive parameter $\eta$, we set the time levels
\begin{align*}
\theta=\eta\rho^{q} \text {\ \ \ when   } \operatorname*{\max}_{Q_{\rho,\eta\rho^{sp}}(x_0,t_0)} a(x,t) \geq 2[a]_{\alpha,\frac{\alpha}{2}} \rho^\alpha,\nonumber\\
\theta=\eta\rho^{sp} \text {\ \ \ when   } \operatorname*{\max}_{Q_{\rho,\eta\rho^{sp}}(x_0,t_0)} a(x,t) \leq 2[a]_{\alpha,\frac{\alpha}{2}} \rho^\alpha.
\end{align*}
Then, for the  lower semicontinuous representative $u_*$ given by Theorem \ref{thm-1-4}, we have
\begin{align*}
u_*(x,t)=\inf_{\eta>0} \lim _{\varrho \rightarrow 0}\operatorname*{ess\inf}_{Q'_{\rho,\theta}(x,t)} u \text{\ \ for every } (x,t)\in Q_T,
\end{align*}
where $Q'_{\rho,\theta}(x,t)=B_\rho(x)\times(t-2\theta,t-\theta)$. In particular, there holds
\begin{align*}
u_*(x,t)=\operatorname*{ess\lim\inf}_{(y,\hat{t}) \rightarrow(x,t),\ \hat{t}<t} u(y,\hat{t}) \text{\ \ for every } (x,t)\in Q_T.
\end{align*}
\end{theorem}

\subsection{Auxiliary lemmas}

We introduce several fundamental lemmas. The following one is the useful Sobolev inequality.
\begin{lemma}[Lemma 2.3, \cite{DZZ21}]
\label{lem-1-8}
Let $0<t_1<t_2$, $s\in(0,1)$ and $p\in(1,\infty)$. Then for every
\begin{align*}
u\in L^p\left(t_1,t_2;W^{s,p}(B_\rho)\right) \cap L^\infty\left(t_1,t_2;L^2(B_\rho)\right),
\end{align*}
we have
\begin{align}
\label{1.15}
&\quad\int_{t_1}^{t_2} \mint_{B_\rho}|u(x,t)|^{p(1+\frac{2s}{N})}\,dxdt \nonumber\\
&\leq C\left(\rho^{sp} \int_{t_1}^{t_2}\int_{B\rho}\mint_{B_\rho} \frac{|u(x,t)-u(y,t)|^p}{|x-y|^{N+sp}}\,dxdydt+\int_{t_1}^{t_2} \mint_{B_\rho}|u(x,t)|^p\,dxdt\right) \nonumber\\
&\quad \times \left(\operatorname*{ess\sup}_{t_1<t<t_2} \mint_{B_\rho}|u(x, t)|^2\,dx\right)^{\frac{sp}{N}},
\end{align}
where $C>0$ depends on $s,p$ and $N$.
\end{lemma}

Next, we give a classical iterative lemma which will be employed to prove the main results.
\begin{lemma} [Lemma 4.1, \cite{D93}]
\label{lem-1-9}
Let $\{Y_j\}_{j=0}^\infty$ be a sequence of positive numbers, satisfying
\begin{align*}
\ Y_{j+1}\leq Kb^jY_j^{1+\delta},\quad j=0,1,2, \ldots
\end{align*}
for some constants $K$, $b>1$ and $\delta>0$. If
\begin{align*}
Y_0\leq K^{-\frac{1}{\delta}}b^{-\frac{1}{\delta^2}},
\end{align*}
then $Y_j\rightarrow 0$ as $j\rightarrow\infty$.
\end{lemma}
The article is organized as follows. Section \ref{sec2} is devoted to proving the necessary Caccioppoli type inequality. In Section \ref{sec3}, we will consider the boundedness estimates by using De Giorgi-Nash-Moser iteration. Finally, the semicontinuity representative and pointwise behavior of weak solutions will be discussed in Section \ref{sec4}.

\section{Energy Estimates}
\label{sec2}
In this section, a Caccioppoli type inequality is established. In general, we need to begin with time regularization since the time derivative of a weak solution to \eqref{1.1} does not exist in the Sobolev sense. The specific procedure can be performed as \cite{KKM22,KKS23}, and thus we omit this step here.

\begin{lemma} [Caccioppoli-type inequality]
\label{lem-2-1}
Let $p,q>1$ and every cylinder $Q^-_{\rho,\theta}\subset Q_T$. Suppose that $u$ is a local weak subsolution to \eqref{1.1}. There exists constants $C>0$ depending on the data such that for every $k\in\mathbb{R}$, there holds
\begin{align}
\label{2.1}
&\quad\operatorname*{ess\sup}_{t_0-\theta<t<t_0}\int_{B_\rho} w_+^2(x,t) \psi^m(x,t)\,dx+\int_{t_0-\theta}^{t_0} \int_{B_\rho} a(x,t)|\nabla w_+(x,t)|^q\psi^m(x,t)\,dxdt\nonumber\\
&\quad+\int_{t_0-\theta}^{t_0} \int_{B_\rho} \int_{B_\rho}(\min\{\psi(x,t),\psi(y,t)\})^m
|w_+(x,t)-w_+(y,t)|^p\, d\mu dt\nonumber\\
&\leq C \int_{t_0-\theta}^{t_0} \int_{B_\rho} w_+^2(x,t)|\partial_t \psi^m|\,dxdt+ C \int_{t_0-\theta}^{t_0} \int_{B_\rho}a(x,t)  w_+^q(x,t)|\nabla \psi(x,t)|^q\,dxdt\nonumber\\
&\quad+ C \int_{t_0-\theta}^{t_0} \int_{B_\rho} \int_{B_\rho} (\max\{w_+(x,t), w_+(y,t)\})^p\left|\psi^{\frac{m}{p}}(x,t)-\psi^{\frac{m}{p}}(y,t)\right|^p \, d\mu dt\nonumber\\
&\quad+C \mathop{\mathrm{ess}\sup}_{\stackrel{t_0-\theta<t<t_0}{x \in \mathrm{supp}\,\psi(\cdot,t)}}  \int_{\mathbb{R}^N \backslash B_\rho} \frac{w_+^{p-1}(y,t)}{|x-y|^{N+sp}}\,dy \int_{t_0-\theta}^{t_0} \int_{B_\rho} w_+(x,t) \psi^m(x,t) \, dxdt,
\end{align}
where $w_+(x,t):=\left(u(x,t)-k\right)_+$, $m=\max\{p,q\}$ and $\psi(x,t)\in[0,1]$ is a piecewise smooth cutoff function vanishing on $\partial B_\rho(x_0)\times(t_0-\theta,t_0)$.
\end{lemma}

\begin{proof}
Denote $w_+(x,t):=\left(u(x,t)-k\right)_+$ and $m=\max\{p,q\}$. Choose the piecewise smooth cutoff function $\psi(\cdot,t)\in[0,1]$ that is compactly supported in $B_\rho$ for any $(t_0-\theta,t_0)$. Testing the weak formulation \eqref{1.9} by $\varphi(x,t):=w_+(x,t)\psi^m(x,t)$, we have
\begin{align*}
0&\geq\int_{t_0-\theta}^{t_0}\int_{B_\rho}\partial_t u(x,t) w_+(x,t)\psi^m(x,t)\,dxdt\\
&\quad+\int_{t_0-\theta}^{t_0}\int_{B_\rho}a(x,t)|\nabla u(x,t)|^{q-2}\nabla u(x,t)\nabla(w_+(x,t)\psi^m(x,t))\,dxdt\\
&\quad+\int_{t_0-\theta}^{t_0}\int_{B_\rho}\int_{B_\rho}J_p(w(x,t),w(y,t))\left(w_+(x,t)\psi^m(x,t)-w_+(y,t)\psi^m(y,t)\right)\,d\mu dt\\
&\quad+2\int_{t_0-\theta}^{t_0}\int_{\mathbb{R}^N\backslash B_\rho}\int_{B_\rho}J_p\left(w(x,t),w(y,t)\right)w_+(x,t)\psi^m(x,t)\,d\mu dt\\
&=:I_1+I_2+I_3+I_4.
\end{align*}	
Then we estimate $I_1, I_2, I_3$ and $I_4$ separately. First, we evaluate
\begin{align}
\label{2.2}
I_1&=\frac{1}{2}\int_{t_0-\theta}^{t_0}\int_{B_\rho}\partial_t w_+^2(x,t)\psi^m(x,t)\,dxdt\nonumber\\
&=\frac{1}{2}\int_{B_\rho}w_+^2(x,t)\psi^m(x,t)\,dx\bigg|_{t_0-\theta}^{t_0}-\int_{t_0-\theta}^{t_0}\int_{B_\rho}w_+^2(x,t)\partial_t\psi^m\,dxdt.
\end{align}
By Young's inequality, the term $I_2$ can be treated as
\begin{align}
\label{2.3}
I_2&=\int_{t_0-\theta}^{t_0}\int_{B_
\rho}a(x,t)|\nabla w_+(x,t)|^{q-2}\nabla w_+(x,t)\nabla(w_+(x,t)\psi^m(x,t))\,dxdt\nonumber\\
&\geq \int_{t_0-\theta}^{t_0}\int_{B_\rho}a(x,t)|\nabla w_+(x,t)|^q\psi^m(x,t)\,dxdt\nonumber\\
&\quad-\int_{t_0-\theta}^{t_0} \int_{B_\rho}a(x,t)m\psi^{m-1}|\nabla\psi(x,t)| |\nabla w_+|^{q-1}w_+(x,t)\,dxdt\nonumber\\
&\geq \int_{t_0-\theta}^{t_0} \int_{B_\rho} a(x,t)|\nabla w_+|^q\psi^m(x,t)\,dxdt-\varepsilon\int_{t_0-\theta}^{t_0}\int_{B_\rho}a(x,t)\psi^{\frac{q(m-1)}{q-1}}|\nabla w_+|^q\,dxdt\nonumber\\
&\quad-C(p,q,\varepsilon)\int_{t_0-\theta}^{t_0}\int_{B_\rho}a(x,t)|\nabla \psi(x,t)|^q w_+^q(x,t)\,dxdt\nonumber\\
&\geq\frac{1}{2}\int_{t_0-\theta}^{t_0}\int_{B_\rho}a(x,t)|\nabla w_+|^q\psi^m\,dxdt-C(p,q)\int_{t_0-\theta}^{t_0}\int_{B_\rho}a(x,t)|\nabla \psi|^q w_+^q(x,t)\,dxdt,
\end{align}
where we used the property $0\leq\psi(x,t)\leq 1$ in the last line. For the term $I_3$, we utilize the idea from Lemma 4.1 in \cite{DM22} to obtain the pointwise estimates. We consider the case $u(x,t)\geq u(y,t)$ with some $t\in(0,T)$, while case $u(x,t)\leq u(y,t)$ can be verified by exchanging the role of $x$ and $y$. Denote
\begin{align*}
\Gamma(x,y,t):=J_p(w(x,t),w(y,t))\left(w_+(x,t)\psi^m(x,t)-w_+(y,t)\psi^m(y,t)\right).
\end{align*}
Since $u(x,t)\geq u(y,t)$, it is not hard to check that $\Gamma_1(x,y,t)+\Gamma_2(x,y,t)\leq \Gamma(x,y,t)$, where
\begin{align*}
\Gamma_1(x,y,t):=|w_+(x,t)-w_+(y,t)|^p\psi^m(x,t)
\end{align*}
and
\begin{align*}
\Gamma_2(x,y,t):=J_p(w_+(x,t),w_+(y,t))(\psi^m(x,t)-\psi^m(y,t))w_+(y,t).
\end{align*}
Moreover, we assume that $\psi(x,t)\geq\psi(y,t)$. By Mean Value Theorem with respect to space variable, we have
\begin{align*}
|\Gamma_2(x,y,t)|\leq C\psi^\frac{m(p-1)}{p}(x,t)\left|\psi^\frac{m}{p}(x,t)-\psi^\frac{m}{p}(y,t)\right||w_+(x,t)-w_+(y,t)|^{p-1}w_+(y,t).
\end{align*}
To estimate $\Gamma_1$, we employ Young's inequality to yield
\begin{align*}
\Gamma_1(x,y,t)\leq C \Gamma(x,y,t)+ C\left|\psi^\frac{m}{p}(x,t)-\psi^\frac{m}{p}(y,t)\right|^p w_+^p(y,t).
\end{align*}
Thus, we conclude that
\begin{align}
\label{2.4}
&\quad|w_+(x,t)-w_+(y,t)|^p\psi^m(x,t)\nonumber\\
&\leq C \Gamma(x,y,t)+ C(\max\{w_+(x,t),w_+(y,t)\})^p\left|\psi^\frac{m}{p}(x,t)-\psi^\frac{m}{p}(y,t)\right|^p,
\end{align}
where constants $C$ only depend on $p$ and $q$. Note that the inequality \eqref{2.4} apparently holds when $\psi(x,t)\leq\psi(y,t)$.
Then, it follows from \eqref{2.4} that
\begin{align}
\label{2.5}
I_3&\geq\frac{1}{C}\int_{t_0-\theta}^{t_0}\int_{B_\rho}\int_{B_\rho}(\min\{\psi(x,t),\psi(y,t)\})^m|w_+(x,t)-w_+(y,t)|^p\,d\mu dt\nonumber\\
&\quad-C\int_{t_0-\theta}^{t_0}\int_{B_\rho}\int_{B_\rho}(\max\{w_+(x,t),w_+(y,t)\})^p\left|\psi^{\frac{m}{p}}(x,t)-\psi^{\frac{m}{p}}(y,t)\right|^p\,d\mu dt.
\end{align}
Finally, we consider the term $I_4$. When $w(x,t)>0$, we have
\begin{align*}
&|w(x,t)-w(y,t)|^{p-2}(w(x,t)-w(y,t))w_+(x,t)\\
\geq&-(w(y,t)-w(x,t))_+^{p-1}w_+(x,t)\\
\geq&-w_+^{p-1}(y,t)w_+(x,t).
\end{align*}
When $w(x,t)\leq 0$, the above inequality still holds due to the both sides are zero. Thus we derive that
\begin{align}
\label{2.6}
I_4\geq& -C \int_{t_0-\theta}^{t_0}\int_{B_\rho}\int_{\mathbb{R}^N \backslash B_\rho}w_+^{p-1}(y,t)w_+(x,t)\psi^m(x,t)\,d\mu dt\nonumber\\
\geq&-C \mathop{\mathrm{ess}\sup}_{\stackrel{t_0-\theta<t<t_0}{x \in \mathrm{supp}\,\psi(\cdot,t)}}\int_{\mathbb{R}^N \backslash B_\rho}\frac{w_+^{p-1}(y,t)}{|x-y|^{N+sp}}\,dy\int_{t_0-\theta}^{t_0}\int_{B_\rho}w_+(x,t)\psi^m(x,t)\,dxdt.
\end{align}
Putting the estimates \eqref{2.2}, \eqref{2.3}, \eqref{2.5} and \eqref{2.6} together, we conclude that
\begin{align*}
&\quad\int_{B_\rho} w_+^2(x,t_0) \psi^m(x,t)\,dx+\int_{t_0-\theta}^{t_0} \int_{B_\rho} a(x,t)|\nabla w_+(x,t)|^q\psi^m(x,t)\,dxdt\nonumber\\
&\quad+\int_{t_0-\theta}^{t_0}\int_{B_\rho}\int_{B_\rho}(\min\{\psi(x,t),\psi(y,t)\})^m|w_+(x,t)-w_+(y,t)|^p\,d\mu dt\nonumber\\
&\leq C \int_{t_0-\theta}^{t_0} \int_{B_\rho} a(x,t)|\nabla \psi(x,t)|^q w_+^q(x,t)\,dxdt\nonumber\\
&\quad+ C \int_{t_0-\theta}^{t_0} \int_{B_\rho} \int_{B_\rho} \max \left\{w_+(x,t), w_+(y,t)\right\}^p\left|\psi^{\frac{m}{p}}(x,t)-\psi^{\frac{m}{p}}(y,t)\right|^p \, d\mu dt\nonumber\\
&\quad+C \mathop{\mathrm{ess}\sup}_{\stackrel{t_0-\theta<t<t_0}{x \in \mathrm{supp}\, \psi(\cdot,t)}} \int_{\mathbb{R}^N \backslash B_\rho} \frac{w_+^{p-1}(y,t)}{|x-y|^{N+sp}}\,dy \int_{t_0-\theta}^{t_0} \int_{B_\rho} w_+(x,t) \psi^m(x,t) \, dxdt\nonumber\\
&\quad+C \int_{t_0-\theta}^{t_0} \int_{B_\rho} w_+^2(x,t)|\partial_t \psi^m|\,dxdt,
\end{align*}
which leads to the desired result.
\end{proof}

\section{Local boundedness}
\label{sec3}
In this section, we aim at investigating the local boundedness of weak solutions. Before doing this, we give some notations as below. 

Let $(x_0,t_0)\in Q_T$, $\rho>0$ and $Q_\rho^-\equiv B_\rho(x_0)\times(t_0-\theta,t_0)$ satisfy $\bar{B}_\rho(x_0) \subseteq \Omega$ and $\left[t_0-\theta,t_0\right] \subseteq(0,T)$. For $\sigma\in[1/2,1)$, we take decreasing sequences
\begin{align*}
\rho_0:=\rho, \quad \rho_j:=\sigma \rho+2^{-j}(1-\sigma) \rho, \quad \tilde{\rho}_j:=\frac{\rho_j+\rho_{j+1}}{2}, \quad j=0,1,2, \ldots
\end{align*}
and
\begin{align*}
\theta_0:=\theta, \quad \theta_j:=\sigma \theta+2^{-j}(1-\sigma) \theta, \quad \tilde{\theta}_j:=\frac{\theta_j+\theta_{j+1}}{2}, \quad j=0,1,2, \ldots.
\end{align*}
Set the cylinders
\begin{align*}
Q_j^-:=B_j \times I_j:=B_{\rho_j}(x_0) \times\left(t_0-\theta_j,t_0\right), \quad j=0,1,2, \ldots,
\end{align*}
and
\begin{align*}
\tilde{Q}_j^-:=\tilde{B}_j\times \tilde{I}_j:=B_{\tilde{\rho}_j}(x_{0}) \times\left(t_0-\tilde{\theta}_j,t_0\right), \quad j=0,1,2, \ldots .
\end{align*}

Denote the sequences of increasing levels
\begin{align*}
k_j:=\left(1-2^{-j}\right)\tilde{k}, \quad \tilde{k}_j:=\frac{k_{j+1}+k_j}{2}, \quad j=0,1,2, \ldots
\end{align*}
with
\begin{align*}
\tilde{k} \geq \frac{\mathrm{Tail}_\infty\left(u_+;x_0,\sigma \rho,t_0-\theta,t_0\right)}{2}.
\end{align*}
Let
\begin{align*}
w_j:=(u-k_j)_+, \quad \tilde{w}_j:=\left(u-\tilde{k}_j\right)_+, \quad j=0,1,2, \ldots.
\end{align*}

The following recursive estimate is obtained by utilizing the Caccioppoli type inequality in a suitable cylinder.

\begin{lemma}
\label{lem-3-1}
Let the assumptions \eqref{1.2} and \eqref{1.3} hold. Suppose that $u$ is a local weak subsolutions to \eqref{1.1}. Let $(x_0, t_0) \in Q_T$, $\rho>0$ and $Q_\rho^-=B_\rho(x_0) \times(t_0-\theta,t_0)$ such that $\bar{B}_{\rho}(x_0) \subseteq \Omega$ and $\left[t_0-\theta,t_0\right] \subseteq(0,T)$. Let $\delta>0$ be a parameter such that $\delta\geq\max\{p,q,2\}$. Then it holds for all $j\in\mathbb{N}$ that
\begin{align}
\label{3.1}
&\quad\operatorname*{ess\sup}_{t\in I_{j+1}}\mint_{B_{j+1}} \tilde{w}_j^2(x,t)\, dx+\int_{I_{j+1}}\mint_{B_{j+1}}a(x,t)|\nabla \tilde{w}_j(x,t)|^q\,dxdt\nonumber\\
&\quad+\int_{I_{j+1}} \int_{B_{j+1}} \mint_{B_{j+1}} \frac{|\tilde{w}_j(x,t)-\tilde{w}_j(y,t)|^p}{|x-y|^{N+sp}}\,dxdydt\nonumber\\
&\leq C\left[\frac{1}{\sigma^{sp}(1-\sigma)^{N+s p}}+\frac{1}{(1-\sigma)^m}\right]\nonumber\\
&\quad\times\left[\frac{2^{\delta j}}{\rho^q\tilde{k}^{\delta-q}}+\frac{2^{(\delta+sp-2)j}}{\theta\tilde{k}^{\delta-2}}+\frac{2^{(\delta+m-p)j}}{\rho^{sp}\tilde{k}^{\delta-p}}+\frac{2^{(\delta+N+m-1)j}}{\rho^{sp}\tilde{k}^{\delta-p}}\right]\int_{I_j} \mint_{B_j} w_j^\delta(x,t)\,dxdt,
\end{align}
where $m=\max\{p,q\}$, the constant $C>0$ only depends on the data, $\delta$ and $\|a\|_{L^\infty(Q_T)}$.
\end{lemma}
\begin{proof}
For every $0<\tau<\delta$, we can check that
\begin{align}
\label{3.2}
\tilde{w}_j^{\tau}(x,t) \leq \frac{C 2^{(\delta-\tau)j}}{\tilde{k}^{\delta-\tau}} w_j^\delta(x,t) \quad \text { in } Q_T.
\end{align}	
Choose the cutoff functions $\psi_j\in C_0^\infty(Q_j)$ such that
\begin{align*}
0 \leq \psi_j \leq 1, \quad\left|\nabla \psi_j\right| \leq \frac{C 2^j}{(1-\sigma) \rho} \text { in } \tilde{B}_j,\quad \left|\partial_t \psi_j\right| \leq \frac{C 2^{spj}}{(1-\sigma)^{sp}\theta} \text { in } \tilde{I}_j,
\quad \psi_j \equiv 1 \text { in } Q_{j+1}.
\end{align*}
An application of Lemma \ref{lem-2-1} with $\rho=\rho_j$, $\theta=\theta_j$ provides that
\begin{align}
\label{3.3}
&\quad\operatorname*{ess\sup}_{t\in I_{j+1}}\int_{B_{j+1}} \tilde{w}_j^2(x,t)\, dx+\int_{I_{j+1}} \int_{B_{j+1}}a(x,t)|\nabla \tilde{w}_j(x,t)|^q\,dxdt\nonumber \\
&\quad+\int_{I_{j+1}} \int_{B_{j+1}}\int_{B_{j+1}}\frac{\left|\tilde{w}_j(x,t) -\tilde{w}_j(y,t) \right|^p}{|x-y|^{N+sp}}
\, dxdy dt\nonumber\\
&\leq C\int_{I_j} \int_{B_j}a(x,t)|\nabla \psi_j(x)|^q \tilde{w}_j^q(x,t)\,dxdt\nonumber\\
&\quad+C\int_{I_j} \int_{B_j} \int_{B_j}(\max\{\tilde{w}_j(x,t), \tilde{w}_j(y,t)\})^p\left|\psi_j^{\frac{m}{p}}(x,t)-\psi_j^{\frac{m}{p}}(y,t)\right|^p \,d\mu dt\nonumber \\
&\quad+C\mathop{\mathrm{ess}\sup}_{\stackrel{t\in I_j}{x \in \mathrm{supp}\, \psi_j(\cdot,t)}}\int_{\mathbb{R}^N\backslash B_j} \frac{\tilde{w}_j^{p-1}(y,t)}{|x-y|^{N+s p}}\,dy\int_{I_j} \int_{B_j}\tilde{w}_j(x,t)\psi_j^m(x,t) \,d xdt\nonumber\\	&\quad+C\int_{I_j} \int_{B_j} \tilde{w}_j^2(x,t)\left|\partial_t \psi^m\right|\,dxdt\nonumber\\
&=: I_1+I_2+I_3+I_4.
\end{align}
By the definition of $\psi_j$ and \eqref{3.2}, we have
	\begin{align}
	\label{3.4}
	I_1&\leq \frac{C2^{qj}\|a\|_{L^\infty(Q_T)} }{(1-\sigma)^{q} \rho^q} \int_{I_j} \int_{B_j} \tilde{w}_j^q(x,t)\,dxdt\nonumber\\
	&\leq \frac{C 2^{\delta j}}{\tilde{k}^{\delta-q}(1-\sigma)^{q} \rho^q} \int_{I_j} \int_{B_j} w_j^\delta(x,t)\,dxdt
	\end{align}
	and
	\begin{align}
	\label{3.5}
	I_2&\leq \frac{C 2^{(m+\delta-p)j}}{\tilde{k}^{\delta-p}(1-\sigma)^{m}\rho^{sp}} \int_{I_j} \int_{B_j} w_j^\delta(x,t)\,dxdt.
	\end{align}
	Observe that
	\begin{align*}
	\frac{|y-x_0|}{|y-x|}\leq 1+\frac{|x-x_0|}{|x-y|}\leq 1+\frac{\tilde{\rho}_j}{\rho_j-\tilde{\rho}_j}\leq 4+\frac{2^{j+2}\sigma}{1-\sigma},
	\end{align*}
	then we obtain
	\begin{align*}
	&\mathop{\mathrm{ess}\sup}_{\stackrel{t\in I_j}{x \in \mathrm{supp}\  \psi_j(\cdot,t)}}\int_{\mathbb{R}^N \backslash B_j}\frac{\tilde{w}_j^{p-1}(y,t)}{|x-y|^{N+s p}}\,dy\\
	\leq&\frac{C 2^{(N+sp)j}}{(1-\sigma)^{N+sp}}\operatorname*{ess\sup}_{t\in I_j}\int_{\mathbb{R}^N \backslash B_j}\frac{\tilde{w}_j^{p-1}(y,t)}{|x_0-y|^{N+s p}}\,dy\\
	\leq&\frac{C 2^{(N+sp)j}}{(1-\sigma)^{N+sp}}\operatorname*{ess\sup}_{t\in I_j}\int_{\mathbb{R}^N \backslash B_{\sigma\rho}}\frac{\tilde{w}_0^{p-1}(y,t)}{|x_0-y|^{N+s p}}\,dy\\
	\leq&\frac{C 2^{(N+sp)j}}{\rho^{sp}\sigma^{sp}(1-\sigma)^{N+sp}}[\mathrm{Tail}_\infty(u_+;x_0,\sigma \rho,t_0-\theta,t_0)]^{p-1},
	\end{align*}
	where we used formula \eqref{1.8}. Recalling the choice of $\tilde{k}$, it can be derived from \eqref{3.2} that
	\begin{align}
	\label{3.6}
	I_3\leq\frac{2^{(N+sp+\delta-1)j}}{(1-\sigma)^{N+sp}\rho^{sp}\sigma^{sp}\tilde{k}^{\delta-p}}\int_{I_j}\int_{B_j} w_j^\delta(x,t)\,dxdt.
	\end{align}
	Moreover, the term $I_4$ can be estimated still by \eqref{3.2} as
	\begin{align}
	\label{3.7}
	I_4&\leq \frac{C 2^{spj}}{(1-\sigma)^{sp} \theta} \int_{I_j} \int_{B_j} w_j^2(x,t)\,dxdt\nonumber\\
	&\leq\frac{C 2^{(sp+\delta-2)j}}{\tilde{k}^{\delta-2}(1-\sigma)^{sp} \theta} \int_{I_j} \int_{B_j} w_j^\delta(x,t)\,dxdt.
	\end{align}
	 We conclude from \eqref{3.3}--\eqref{3.7} that
	\begin{align*}
	&\operatorname*{ess\,\sup}_{t\in I_{j+1}}\mint_{B_{j+1}} \tilde{w}_j^2(x,t)\,dx+\int_{I_{j+1}}\mint_{B_{j+1}}a(x,t)|\nabla \tilde{w}_j(x,t)|^q\,dxdt\\
	&+\int_{I_{j+1}} \int_{B_{j+1}} \mint_{B_{j+1}} \frac{|\tilde{w}_j(x,t)-\tilde{w}_j(y,t)|^p}{|x-y|^{N+sp}}\,dxdydt\nonumber\\
	\leq& C \left[\frac{1}{(1-\sigma)^q}+\frac{1}{(1-\sigma)^m}+\frac{1}{(1-\sigma)^{sp}}+\frac{1}{(1-\sigma)^{N+sp}\sigma^{sp}}\right]\\ &\times\left[\frac{2^{\delta j}}{\rho^q\tilde{k}^{\delta-q}}+\frac{2^{(m+\delta-p)j}}{\rho^{sp}\tilde{k}^{\delta-p}}+\frac{2^{(N+sp+\delta-1)j}}{\rho^{sp}\tilde{k}^{\delta-p}}+\frac{2^{(sp+\delta-2)j}}{\theta\tilde{k}^{\delta-2}}\right]\int_{I_j} \mint_{B_j} w_j^\delta(x,t)\,dxdt,
	\end{align*}
	as desired.
\end{proof}


\begin{lemma}
	\label{lem-3-2}
    Let the assumptions \eqref{1.2} and \eqref{1.3} hold, $p>2N/(N+2s), \kappa=1+2s/N$ and $q<p\kappa$. Suppose that $u$ is a local weak subsolution to \eqref{1.1}. Let $(x_0,t_0) \in Q_T$, $\rho>0$ and $Q_\rho^-=B_\rho(x_0) \times(t_0-\theta,t_0)$ satisfy $\bar{B}_\rho(x_0) \subseteq \Omega$ and $\left[t_0-\theta,t_0\right] \subseteq(0,T)$. Let $\delta$ be a parameter such that $\max\{p,q,2\}\leq\delta\leq\kappa p$. Then we infer for all $j\in \mathbb{N}$ that
	\begin{align}
	\label{3.8}
	\frac{1}{\rho^{s p}} \int_{I_{j+1}} \mint_{B_{j+1}} w_{j+1}^\delta(x, t) d x d t \leq C 2^{bj} \tilde{B}\left(\frac{A_k}{\rho^{sp}} \int_{I_j} \mint_{ B_j} w_j^\delta d x d t\right)^{1+\frac{s \delta}{N \kappa}},
	\end{align}
	where $C>0$ depends on $\|a\|_{L^\infty(Q_T)},\delta$, the data and
	\begin{align*}
	&b=\left(1+\frac{s p}{N}\right)\left(N+m+\delta\right), \quad A_k  :=\frac{1}{\tilde{k}^{\delta-p}}+\frac{\rho^{sp-q }}{\tilde{k}^{\delta-q}}+\frac{\rho^{s p}}{\theta \tilde{k}^{\delta-2}},\\
	 &\tilde{B}=\left[\frac{1}{\sigma^{p s}(1-\sigma)^{N+p s}}+\frac{1}{(1-\sigma)^m}\right]^{\left(1+\frac{s p}{N}\right)\frac{\delta}{p \kappa}} .
	\end{align*}
\end{lemma}
    \begin{proof}
    Since $\delta\leq p\kappa$, we derive from H\"{o}lder inequality that
    \begin{align}
    \label{3.9}
    &\quad\int_{I_{j+1}} \mint_{B_{j+1}} w_{j+1}^\delta(x,t)\,dxdt \nonumber\\
    &\leq \int_{I_{j+1}} \mint_{B_{j+1}} \tilde{w}_j^\delta(x,t)\,dxdt\nonumber\\
    &\leq\left(\int_{I_{j+1}} \mint_{B_{j+1}} \tilde{w}_j^{p \kappa}(x,t)\,dxdt\right)^{\frac{\delta}{p\kappa}}\left(\int_{I_{j+1}} \mint_{B_{j+1}} \chi_{\{u \geq \tilde{k}_j\}}(x,t)\,dxdt\right)^{1-\frac{\delta}{p \kappa}}.
    \end{align}
    In view of \eqref{3.2}, we can directly get
    \begin{align}
    \label{3.10}
    \int_{I_{j+1}} \mint_{B_{j+1}} \chi_{\{u \geq \tilde{k}_j\}}(x,t)\,dxdt\leq \frac{C 2^{\delta j}}{\tilde{k}^\delta} \int_{I_j} \mint_{B_j} w_j^\delta(x,t)\,dxdt
    \end{align}
    and
    \begin{align}
    \label{3.11}
    \int_{I_{j+1}} \mint_{B_{j+1}} \tilde{w}_j^p(x,t)\,dxdt \leq \frac{C 2^{(\delta-p)j}}{\tilde{k}^{\delta-p}} \int_{I_j} \mint_{B_j} w_j^\delta(x,t)\,dxdt.
    \end{align}	
    Taking into account \eqref{3.11}, Lemma \ref{lem-1-8} and Lemma \ref{lem-3-1}, there holds that
    \begin{align}
    \label{3.12}
    &\quad\int_{I_{j+1}}\mint_{B_{j+1}}\tilde{w}_j^{p\kappa}(x,t)\,dxdt\nonumber\\
    &\leq C \rho^{sp}
    \left(\int_{I_{j+1}} \mint_{B_{j+1}} \mint_{B_{j+1}} \frac{|\tilde{w}_j(x,t)-\tilde{w}_j(y,t)|^p}{|x-y|^{N+sp}}
    \,dxdydt
    +\frac{1}{\rho^{sp}}\int_{I_{j+1}} \mint_{B_{j+1}} \tilde{w}_j^p(x,t)\,dxdt\right)\nonumber\\
    &\quad\times \left(\operatorname*{ess\sup}_{t\in I_{j+1}} \mint_{B_{j+1}} \tilde{w}_j^2(x,t)\,dx\right)^{\frac{sp}{N}}\nonumber\\
    &\leq C \rho^{sp}\left[\frac{2^{\delta j}}{\rho^q\tilde{k}^{\delta-q}}+\frac{2^{(sp+\delta-2)j}}{\theta\tilde{k}^{\delta-2}}+\frac{2^{(N+m+\delta-1)}}{\rho^{sp}\tilde{k}^{\delta-p}}\right]^{1+\frac{sp}{N}}
    \times\left(B\int_{I_j}\mint_{ B_j} w_j^\delta \,dxdt\right)^{1+\frac{sp}{N}}\nonumber\\
    &=C \rho^{sp}\left[\frac{2^{\delta j}\rho^{sp-q}}{\tilde{k}^{\delta-q}}+\frac{2^{(sp+\delta-2)j}\rho^{sp}}{\theta\tilde{k}^{\delta-2}}+\frac{2^{(N+m+\delta-1)}}{\tilde{k}^{\delta-p}}\right]^{1+\frac{sp}{N}}
    \times\left(\frac{B}{\rho^{sp}}\int_{I_j}\mint_{ B_j} w_j^\delta \,dxdt\right)^{1+\frac{sp}{N}}\nonumber\\
    &\leq C 2^{bj}\rho^{sp} \left(\frac{B A_k}{\rho^{sp}}\int_{I_j}\mint_{ B_j} w_j^\delta\,dxdt\right)^{1+\frac{sp}{N}},
    \end{align}	
    where
    \begin{align*}
    &b=\left(1+\frac{s p}{N}\right)\left(N+m+\delta\right), \quad A_k  :=\frac{1}{\tilde{k}^{\delta-p}}+\frac{\rho^{sp-q }}{\tilde{k}^{\delta-q}}+\frac{\rho^{s p}}{\theta \tilde{k}^{\delta-2}},\\
    &B=\frac{1}{\sigma^{sp}(1-\sigma)^{N+sp}}+\frac{1}{(1-\sigma)^m}.
    \end{align*}
     Substituting \eqref{3.10} and \eqref{3.12} into \eqref{3.9}, we yield the claim \eqref{3.8} with $\tilde{B}=B^{\left(1+\frac{s p}{N}\right)\frac{\delta}{p \kappa}}$.
    \end{proof}

\begin{lemma}
	\label{lem-3-3}
	Let  the assumptions \eqref{1.4} and \eqref{1.5} hold, $p\leq 2N/(N+2s)$, $\kappa=1+2s/N$, $q<p\kappa$ and $r>\max\left\{2,\frac{N(2-p)}{sp}\right\}$. Suppose that $u\in L_{\rm{loc}}^\infty(Q_T)$ be a local weak subsolution to \eqref{1.1}. Let $(x_0,t_0) \in Q_T$, $\rho>0$ and $Q_\rho^-=B_\rho(x_0) \times(t_0-\theta,t_0)$ such that $\bar{B}_\rho(x_0) \subseteq \Omega$ and $\left[t_0-\theta,t_0\right] \subseteq(0,T)$. For any $j\in \mathbb{N}$, we have
	
	\begin{align}
	\label{3.13}
	&\frac{1}{\rho^{sp}}\int_{I_{j+1}}\mint_{B_{j+1}}w_{j+1}^r(x,t)\,dxdt\nonumber\\
	\leq& C 2^{b'j}\left\|\tilde{w}_j\right\|_{L^\infty\left(Q_{j+1}^-\right)}^{r-p\kappa}\left(\frac{A'_k B}{\rho^{sp}}\int_{I_j}\mint_{B_j} w_j^r\,dxdt\right)^{1+\frac{sp}{N}},
	\end{align}
	where $C>0$ only depends on $r,\|a\|_{L^\infty(Q_T)}$, the data and
	\begin{align*}
	&b':=\left(1+\frac{sp}{N}\right)\left(N+m+r\right), \ A'_k  :=\frac{1}{\tilde{k}^{r-p}}+\frac{\rho^{sp-q }}{\tilde{k}^{r-q}}+\frac{\rho^{s p}}{\theta \tilde{k}^{r-2}},\\
	& B=\frac{1}{\sigma^{sp}(1-\sigma)^{N+sp}}+\frac{1}{(1-\sigma)^m}.
	\end{align*}
\end{lemma}

\begin{proof}
	Thanks to the assumptions, we can see
	\begin{align}
	\label{3.14}
	\int_{I_{j+1}} \mint_{B_{j+1}} w_{j+1}^r(x,t)\,dxdt& \leq \int_{I_{j+1}} \mint_{B_{j+1}} \tilde{w}_j^r(x,t)\,dxdt\nonumber \\
	& \leq\left\|\tilde{w}_j\right\|_{L^\infty\left(Q_{j+1}^-\right)}^{r-p\kappa} \int_{I_{j+1}} \mint_{B_{j+1}} \tilde{w}_j^{p\kappa}(x,t)\,dxdt
	\end{align}
	with $\kappa=1+2s/N$. Similar to the proof of Lemma \ref{lem-3-2} with $\delta=r$, it gives that
	\begin{align}
	\label{3.15}
	&\quad\int_{I_{j+1}} \mint_{B_{j+1}} \tilde{w}_j^{p\kappa}(x,t)\,dxdt\nonumber\\
	&\leq C \rho^{sp}
	\left(\int_{I_{j+1}} \mint_{B_{j+1}} \mint_{B_{j+1}} \frac{|\tilde{w}_j(x,t)-\tilde{w}_j(y,t)|^p}{|x-y|^{N+sp}}
	\,dxdydt
	+\frac{1}{\rho^{sp}}\int_{I_{j+1}} \mint_{B_{j+1}} \tilde{w}_j^p(x,t)\,dxdt\right)\nonumber\\
	&\quad\times \left(\operatorname*{ess\,\sup}_{t\in I_{j+1}} \mint_{B_{j+1}} \tilde{w}_j^2(x,t)\,dx\right)^{\frac{sp}{N}}\nonumber\\
	&\leq C 2^{b'j}\rho^{sp}\left(\frac{A'_k B}{\rho^{sp}}\int_{I_j}\mint_{B_j} w_j^r\,dxdt\right)^{1+\frac{sp}{N}},
	\end{align}
	where $b', A'_k, B$ are defined as in the statement of claim. Thus, we get the conclusion by merging inequalities \eqref{3.14} and \eqref{3.15}.
\end{proof}

\begin{remark}
	Due to the fact that $\sigma\in[1/2,1)$, we can rewrite \eqref{3.13} as
	\begin{align*}
	&\quad\frac{1}{\rho^{sp}}\int_{I_{j+1}}\mint_{B_{j+1}}w_{j+1}^r(x,t)\,dxdt\nonumber\\
	&\leq C 2^{b'j}\left\|\tilde{w}_j\right\|_{L^\infty\left(Q_{j+1}^-\right)}^{r-p\kappa}\frac{1}{(1-\sigma)^\frac{(N+sp)(N+m)}{N}}\left(\frac{A'_k}{\rho^{sp}}\int_{I_j}\mint_{B_j} w_j^r\,dxdt\right)^{1+\frac{sp}{N}}.
	\end{align*}
\end{remark}
Now we are beginning to prove the boundedness results.

\begin{proof}[\rm{\textbf{Proof of Theorem \ref{thm-1-2}}}]
	Let $\rho=R$, $\sigma=\frac{1}{2}$, then $\rho_j=\frac{R}{2}+2^{-j-1}R$. For $\delta=\max\{2,p,q\}$, set
	\begin{align*}
	Y_j=\frac{1}{\rho^{sp}}\int_{I_j} \mint_{B_j}w_j^\delta\,dxdt, \quad j=0,1,2, \ldots.
	\end{align*}
	It follows from Lemma \ref{lem-3-2} that
	\begin{align}
	\label{3.16}
	\frac{1}{\rho^{sp}}\int_{I_{j+1}} \mint_{B_{j+1}} w_{j+1}^\delta(x,t)\,dxdt\leq C 2^{bj}\left(\frac{A_k}{\rho^{sp}}\int_{I_j}\mint_{B_j}w_j^\delta\,dxdt\right)^{1+\frac{s\delta}{N\kappa}}
	\end{align}
	with $b=\left(1+\frac{s p}{N}\right)\left(N+m+\delta\right), m=\max\{p,q\}$ and $\kappa:=1+2s/N$.
	If assuming $\tilde{k}\geq 1$ and taking $\tau=\min\{\delta-q,\delta-p,\delta-2\}$, it shows that
	\begin{align}
	\label{3.17}
	 A_k=\frac{1}{\tilde{k}^{\delta-p}}+\frac{\rho^{sp-q }}{\tilde{k}^{\delta-q}}+\frac{\rho^{s p}}{\theta \tilde{k}^{\delta-2}}\leq \frac{A}{\tilde{k}^\tau},
	\end{align}
	where
	\begin{align*}
	A=1+\rho^{sp-q}+\frac{\rho^{sp}}{\theta}.
	\end{align*}
	By using \eqref{3.16} and \eqref{3.17}, we have
	\begin{align*}
	Y_{j+1}\leq C 2^{bj}\left(\frac{A}{\tilde{k}^\tau}Y_j\right)^{1+\frac{s\delta}{N\kappa}}.
	\end{align*}
	Choose $\theta=\rho^{sp}$. Let $\tilde{k}$ be taken to satisfy
	\begin{align*}
	\tilde{k} \geq \max \bigg\{\mathrm{Tail}_\infty\left(u_+;x_{0},R/2,t_0-R^{sp}, t_0\right),
	C
	A^{\frac{1}{\tau}}\left(\mint^{t_0}_{t_0-R^{sp}} \mint_{B_R} u_+^\delta\,dxdt\right)^{\frac{s\delta}{\tau(N\kappa+s\delta)}}\vee 1\bigg\},
	\end{align*}
	where $C$ only depends on \textit{the data} and $\|a\|_{L^\infty(Q_T)}$. Applying Lemma \ref{lem-1-9}, we obtain $\lim_{j\rightarrow\infty}Y_j=0$, which implies that
	\begin{align*}
	\operatorname*{ess\,\sup}_{Q^-_{R/2,R^{sp}/2}}u \leq \mathrm{Tail}_\infty\left(u_+;x_0,R/2,t_0-R^{sp}, t_0\right)+CA^{\frac{1}{\tau}}\left(\mint^{t_0}_{t_0-R^{sp}} \mint_{B_R} u_+^\delta\,dxdt\right)^{\frac{s\delta}{\tau(N\kappa+s\delta)}}\vee 1,
	\end{align*}
	as intended.
\end{proof}

\begin{proof}[\rm{\textbf{Proof of Theorem \ref{thm-1-3}}}]
Let the conditions in Theorem \ref{thm-1-3} hold. By means of the appropriate approximation method, we may suppose that $u$ is qualitatively locally bounded based on \eqref{1.10} and \eqref{1.11}. In fact, since the approximation subsolutions $u_k$ are bounded, the following estimate still holds if we substitute $u$ by $u_k$, and combining with \eqref{1.10} and \eqref{1.11} gives a $k$-independent bound on $u_k$ in $L^\infty$. Therefore, we obtain $u$ is  qualitatively locally bounded by the a.e. convergence of $u_k$.

Let $\theta=R^{sp}$, $\theta_n=R_n^{sp}$, $R_0=R/2$ and $R_n=R/2+\sum_{i=1}^n 2^{-i-1}R$ for $n\in\mathbb{N}^+$, and $Q_n^-:=B_{R_n}(x_0) \times\left(t_0-\theta_n, t_0\right)=B_{R_n}(x_0) \times\left(t_0-R_n^{sp}, t_0\right)$.	
Denote
\begin{align*}
M_n=\operatorname{ess} \sup _{Q_n^-} u_+, \quad n=0,1,2,3, \ldots.
\end{align*}
Taking $\rho=R_{n+1}$ and $\sigma \rho=R_{n}$, then
\begin{align*}
\sigma=\frac{1 / 2+\sum_{i=1}^{n} 2^{-i-1}}{1 / 2+\sum_{i=1}^{n+1} 2^{-i-1}} \geq \frac{1}{2}.
\end{align*}	
Set
\begin{align*}
Y_j=\frac{1}{\rho^{sp}}\int_{I_j} \mint_{B_j}\left(u-k_j\right)_+^r\,dxdt, \quad j=0,1,2, \ldots.	
\end{align*}	
In light of Lemma \ref{lem-3-3}, we have
\begin{align*}
Y_{j+1}\leq& C 2^{b'j}\left(2+R^{sp-q}\right)^{1+\frac{sp}{N}}\left\|u_+\right\|_{L^\infty\left(Q_{n+1}^-\right)}^{r-p\kappa}
\frac{1}{(1-\sigma)^\frac{(N+sp)(N+m)}{N}}\\
&\times \left(\frac{1}{\tilde{k}^{r-2}}+\frac{1}{\tilde{k}^{r-p}}+\frac{1}{\tilde{k}^{r-q}}\right)^{1+\frac{sp}{N}}\left(\int_{I_j}\mint_{B_j} w_j^r\,dxdt\right)^{1+\frac{sp}{N}}\\
\leq& C2^{b'j+dn}A^{1+\frac{sp}{N}}M_{n+1}^{r-p\kappa}\left(\frac{1}{\tilde{k}^{r-2}}+\frac{1}{\tilde{k}^{r-p}}+\frac{1}{\tilde{k}^{r-q}}\right)^{1+\frac{sp}{N}} Y_{j}^{1+\frac{sp}{N}}
\end{align*}
with $b':=\left(1+\frac{sp}{N}\right)\left(N+m+r\right)$, $d=(N+sp)(N+m)/N$ and $A=2+R^{sp-q}$.
Then from Lemma \ref{lem-1-9} we get $\lim_{j\rightarrow\infty}Y_j=0$ if
\begin{align}
\label{3.18}
Y_0\leq C 2^{-\frac{dnN}{sp}-\frac{b'N^2}{(sp)^2}}A^{-\frac{N+sp}{sp}}
 M_{n+1}^{-\frac{N(r-p \kappa)}{sp}}\left(\frac{1}{\tilde{k}^{r-2}}+\frac{1}{\tilde{k}^{r-p}}+\frac{1}{\tilde{k}^{r-q}}\right)^{-\frac{N+sp}{sp}},
\end{align}
that is
\begin{align*}
&C 2^{\frac{dnN}{N+sp}}A M_{n+1}^{\frac{N(r-p\kappa)}{N+sp}}\left(\mint_{t_0-R_{n+1}^{sp}}^{t_0}\mint_{B_{R_{n+1}}}u_+^r\,dxdt\right)^{\frac{sp}{N+sp}}\\
\leq&\left(\frac{1}{\tilde{k}^{r-p}}+\frac{1}{\tilde{k}^{r-q}}+\frac{1}{\tilde{k}^{r-2}}\right)^{-1}.
\end{align*}
To ensure the inequality \eqref{3.18} holds true, we can take
\begin{align*}
\tilde{k}&=C 2^{\frac{dnN}{(N+sp)(r-p)}}A^{\frac{1}{r-p}}M_{n+1}^{\frac{N(r-p\kappa)}{(N+sp)(r-p)}}\left(\mint_{t_0-R_{n+1}^{sp}}^{t_0}\mint_{B_{R_{n+1}}}u_+^r\,dxdt\right)^{\frac{sp}{(N+sp)(r-p)}}\\
&\quad+C 2^{\frac{dnN}{(N+sp)(r-q)}}A^{\frac{1}{r-q}}M_{n+1}^{\frac{N(r-p\kappa)}{(N+sp)(r-q)}}\left(\mint_{t_0-R_{n+1}^{sp}}^{t_0}\mint_{B_{R_{n+1}}}u_+^r\,dxdt\right)^{\frac{sp}{(N+sp)(r-q)}}\\
&\quad+C 2^{\frac{dnN}{(N+sp)(r-2)}}A^{\frac{1}{r-2}}M_{n+1}^{\frac{N(r-p\kappa)}{(N+sp)(r-2)}}\left(\mint_{t_0-R_{n+1}^{sp}}^{t_0}\mint_{B_{R_{n+1}}}u_+^r\,dxdt\right)^{\frac{sp}{(N+sp)(r-2)}}\\
&\quad+\frac{\mathrm{Tail}_\infty\left(u_+;x_0,R_n,t_0-R_{n+1}^{sp}, t_0\right)}{2},
\end{align*}
where constants $C$ depend on \textit{the data}, $r$ and $\|a\|_{L^\infty(Q_T)}$. Thus by Lemma \ref{lem-1-9} we obtain
\begin{align}
\label{3.19}
M_n&=\operatorname*{ess\sup}_{Q_{R_n}^-} u^+\nonumber\\
&\leq C 2^{\frac{dnN}{(N+sp)(r-p)}}A^{\frac{1}{r-p}}M_{n+1}^{\frac{N(r-p\kappa)}{(N+sp)(r-p)}}\left(\mint_{t_0-R_{n+1}^{sp}}^{t_0}\mint_{B_{R_{n+1}}}u_+^r\,dxdt\right)^{\frac{sp}{(N+sp)(r-p)}}\nonumber\\
&\quad+C 2^{\frac{dnN}{(N+sp)(r-q)}}A^{\frac{1}{r-q}}M_{n+1}^{\frac{N(r-p\kappa)}{(N+sp)(r-q)}}\left(\mint_{t_0-R_{n+1}^{sp}}^{t_0}\mint_{B_{R_{n+1}}}u_+^r\,dxdt\right)^{\frac{sp}{(N+sp)(r-q)}}\nonumber\\
&\quad+C 2^{\frac{dnN}{(N+sp)(r-2)}}A^{\frac{1}{r-2}}M_{n+1}^{\frac{N(r-p\kappa)}{(N+sp)(r-2)}}\left(\mint_{t_0-R_{n+1}^{sp}}^{t_0}\mint_{B_{R_{n+1}}}u_+^r\,dxdt\right)^{\frac{sp}{(N+sp)(r-2)}}\nonumber\\
&\quad+\frac{\mathrm{Tail}_\infty\left(u_+;x_0,R_n,t_0-R_{n+1}^{sp}, t_0\right)}{2}.
\end{align}
In view of the assumption $r>\max\left\{2,\frac{N(2-p)}{sp}\right\}$, which indicates that
\begin{align*}
0<\frac{N(r-p\kappa)}{(N+sp)(r-p)},\frac{N(r-p\kappa)}{(N+sp)(r-q)},\frac{N(r-p\kappa)}{(N+sp)(r-2)}<1.
\end{align*}
Applying Young's inequality to \eqref{3.19}, we get
\begin{align*}
M_{n} &\leq  \varepsilon M_{n+1}+C 2^{\frac{dnN}{(N+sp)(r-p-\beta)}}\varepsilon^{-\frac{\beta}{r-p-\beta}}A^{\frac{1}{r-p-\beta}} \left(\mint_{t_0-R^{sp}}^{t_0}\mint_{B_R}u_+^r\,dxdt\right)^{\frac{sp}{(N+sp)(r-p-\beta)}}\\
&\quad+C 2^{\frac{dnN}{(N+sp)(r-q-\beta)}}\varepsilon^{-\frac{\beta}{r-q-\beta}}A^{\frac{1}{r-q-\beta}}\left(\mint_{t_0-R^{sp}}^{t_0}\mint_{B_R}u_+^r\,dxdt\right)^{\frac{sp}{(N+sp)(r-q-\beta)}}\\
&\quad+C 2^{\frac{dnN}{(N+sp)(r-2-\beta)}}\varepsilon^{-\frac{\beta}{r-2-\beta}}A^{\frac{1}{r-2-\beta}}\left(\mint_{t_0-R^{sp}}^{t_0}\mint_{B_R}u_+^r\,dxdt\right)^{\frac{sp}{(N+sp)(r-2-\beta)}}\\
&\quad+\frac{\mathrm{Tail}_\infty\left(u_+;x_0,R/2,t_0-R^{sp},t_0\right)}{2}
\end{align*}
with $\beta=(r-p\kappa)N/(sp+N)$. By the induction argument, we deduce that
\begin{align}
\label{3.20}
M_0&\leq \varepsilon^{n+1}M_{n+1}\nonumber\\
&\quad+C\varepsilon^{-\frac{\beta}{r-p-\beta}}A^{\frac{1}{r-p-\beta}}\left(\mint_{t_0-R^{sp}}^{t_0}\mint_{B_R}u_+^r\,dxdt\right)^{\frac{sp}{(N+sp)(r-p-\beta)}}\sum_{i=0}^n \left(2^{{\frac{dN}{(N+sp)(r-p-\beta)}}}\varepsilon\right)^i\nonumber\\
&\quad+C\varepsilon^{-\frac{\beta}{r-q-\beta}}A^{\frac{1}{r-q-\beta}}\left(\mint_{t_0-R^{sp}}^{t_0}\mint_{B_R}u_+^r\,dxdt\right)^{\frac{sp}{(N+sp)(r-q-\beta)}}\sum_{i=0}^n \left(2^{{\frac{dN}{(N+sp)(r-q-\beta)}}}\varepsilon\right)^i\nonumber\\
&\quad+C\varepsilon^{-\frac{\beta}{r-2-\beta}}A^{\frac{1}{r-2-\beta}}\left(\mint_{t_0-R^{sp}}^{t_0}\mint_{B_R}u_+^r\,dxdt\right)^{\frac{sp}{(N+sp)(r-2-\beta)}}\sum_{i=0}^n \left(2^{{\frac{dN}{(N+sp)(r-2-\beta)}}}\varepsilon\right)^i\nonumber\\
&\quad+\frac{\mathrm{Tail}_\infty\left(u_+;x_0,R/2,t_0-R^{sp},t_0\right)}{2}\sum_{i=0}^n\varepsilon^i, \quad n=0,1,2,\ldots.
\end{align}
Observe that the sum on the right-hand side of \eqref{3.20} can be revised by a convergent series if we choose
\begin{align*}
\varepsilon=2^{-\left[{\frac{dN}{(N+sp)(r-2-\beta)}}+1\right]}.
\end{align*}	

To complete the proof, we let $n\rightarrow\infty$ and arrive at
\begin{align*}
\operatorname*{ess\sup}_{Q_{R/2,R^{sp}/2}^-} u&\leq \mathrm{Tail}_\infty\left(u_+;x_0,R/2,t_0-R^{sp}, t_0\right)\\
&\quad+C\left(R^{sp-q}+2\right)^{\frac{1}{r-2-\beta}}\left(\mint_{t_0-R^{sp}}^{t_0}\mint_{B_R} u_+^r\,dx dt\right)^{\frac{sp}{(N+sp)(r-2-\beta)}} \\
&\quad\quad\vee \left(R^{sp-q}+2\right)^{\frac{1}{r-p-\beta}}\left(\mint_{t_0-R^{sp}}^{t_0}\mint_{B_R} u_+^r\,dx dt\right)^{\frac{sp}{(N+sp)(r-p-\beta)}}\\
&\quad\quad \vee\left(R^{sp-q}+2\right)^{\frac{1}{r-q-\beta}} \left(\mint_{t_0-R^{sp}}^{t_0}\mint_{B_R} u_+^r\,dx dt\right)^{\frac{sp}{(N+sp)(r-q-\beta)}},
\end{align*}
where $C>0$ depends only on \textit{the data}, $r$ and $\|a\|_{L^\infty(Q_T)}$.
\end{proof}

\section{Lower semicontinuity and pointwise behavior}
\label{sec4}

In this section, we study the lower semicontinuity and pointwise behavior of weak supersolutions to \eqref{1.1}. We begin with introducing the measure theoretic property of weak solutions. Then we establish the De Giorgi type lemma \ref{lem-4-2}, which is a key ingredient to obtain the Theorem \ref{thm-1-4}.

What follows is the property $(\mathcal{D})$ of weak solutions that can be found in \cite{L21}. Assume that $u$ is a measurable function, locally bounded from below in $Q_T$. Let 
the parameters $0<b,c<1, L>0$ and the cylinders $Q_{\rho,\theta}(x_0,t_0)\subset Q_T$. Let the number $\mu^-$ satisfy
\begin{align*}
\mu^-\leq\operatorname*{ess\inf}_{Q_{\rho,\theta}(x_0,t_0)} u.
\end{align*}
We mean the function $u$ satisfies the property $(\mathcal{D})$ if there exists a constant $\nu\in(0,1)$ only depending on $b, L, \mu^-$ and other data, but independent of $\rho$ such that if
\begin{align*}
|[u\leq\mu^-+L]\cap Q_{\rho,\theta}(x_0,t_0)|\leq \nu| Q_{\rho,\theta}(x_0,t_0)|,
\end{align*}
then
\begin{align*}
u\geq\mu^-+bL \text{ \ \ a.e.\ in\ } Q_{c\rho,c\theta}(x_0,t_0).
\end{align*}

Similar to the Theorem 2.1 in \cite{L21}, we present the following result.

\begin{theorem}
	\label{thm-4-1}
	Let $u$ be a locally essentially bounded below measurable function in $Q_T$ that satisfies the property $(\mathcal{D})$, then we have $u_*(x,t)=u(x,t)$ for every $(x,t)\in\mathcal{F}$. In particular, $u_*$ is a lower semiconutinuous representative of $u$ in $Q_T$.
\end{theorem}

Next, we devote to verifying the weak solutions of \eqref{1.1} fulfill the property $(\mathcal{D})$, which together with Theorem \ref{thm-4-1} leads to the conclusion of Theorem \ref{thm-1-4}.  
\begin{lemma}[De Giorgi type lemma]
\label{lem-4-2}
Let the assumptions \eqref{1.2} and \eqref{1.12} hold. Set $M:=\sup_{Q_T} |u|$ and fix the radius $\rho\in(0,1]$, the constants $L\in(0,M)$, $b\in(0,1]$. Assume that $u$ is a local weak supersolution to \eqref{1.1}. For a positive parameter $\eta$, we set the time levels
\begin{align}
\label{4.1}
\theta=\eta\rho^q \text {\ \ \ when   } \operatorname*{\max}_{Q_{\rho,\eta\rho^{sp}}(x_0,t_0)} a(x,t) \geq 2[a]_{\alpha,\frac{\alpha}{2}} \rho^\alpha,\nonumber\\
\theta=\eta\rho^{sp} \text {\ \ \ when   } \operatorname*{\max}_{Q_{\rho,\eta\rho^{sp}}(x_0,t_0)} a(x,t) \leq 2[a]_{\alpha,\frac{\alpha}{2}} \rho^\alpha.
\end{align}
Let $Q_{\rho,\theta}^-(x_0,t_0)\subset Q_T$, and the number $\mu^-$ satisfy
\begin{align}
\label{4.2}
\mu^-\leq \operatorname*{ess\inf}_{\mathbb{R}^N\times(0,T)} u.
\end{align}
There exists a constant $\nu\in(0,1)$ depending on the data, $M,L,b,\eta,\mu^-$ and $[a]_{\alpha,\frac{\alpha}{2}}$ such that if
\begin{align*}
|[u\leq \mu^-+L]\cap Q^-_{\rho,\theta}(x_0,t_0)|\leq \nu|Q^-_{\rho,\theta}(x_0,t_0)|,
\end{align*}
then
\begin{align*}
u\geq \mu^-+bL \text{\ a.e. in }  Q_{\frac{\rho}{2},\frac{\theta}{2}}(x_0,t_0).
\end{align*}
\end{lemma}
\begin{proof}
We may assume $(x_0,t_0)=(0,0)$ without loss of generality. Let
\begin{gather*}
k_j=\mu^-+bL+\frac{(1-b)L}{2^j}, \ w_j=(k_j-u)_+,\nonumber\\
\rho_j=\rho+\frac{\rho}{2^j},\ \tilde{\rho}_j=\frac{\rho_j+\rho_{j+1}}{2},\  \hat{\rho}_j=\frac{3\rho_j+\rho_{j+1}}{4},\nonumber\\
\theta_j=\theta+\frac{\theta}{2^j},\ \tilde{\theta}=\frac{\theta_j+\theta_{j+1}}{2},\  \hat{\theta_j}=\frac{3\theta_j+\theta_{j+1}}{4},\nonumber\\
B_j=B_{\rho_j},\ \tilde{B}_{j}=\tilde{B}_{\rho_j},\ \hat{B}_{j}=\hat{B}_{r_j},\nonumber\\
Q_j=B_j\times(-\theta_j,0),\ \tilde{Q}_j=\tilde{B}_j\times(-\tilde{\theta}_j,0),\ \hat{Q}_j=\hat{B}_j\times(-\hat{\theta}_j,0).
\end{gather*}
Obviously, there holds
\begin{align*}
Q_{j+1}\subset\tilde{Q}_j\subset\hat{Q}_j\subset Q_j.
\end{align*}
Choose the cutoff functions $\zeta\in C_0^\infty(Q_j)$ vanishing outside $\hat{Q}_j$, and satisfy
\begin{align*}
\zeta=1 \text{\ in } \tilde{Q}_j,\  |\nabla\zeta|\leq\frac{2^{j}}{\rho}, \quad|\partial_t\zeta|\leq \frac{2^{spj}}{\theta}.
\end{align*}

Let
$$\bar{a}_{\rho}:=\operatorname*{\max}_{Q_{\rho,\eta\rho^{sp}}(x_0,t_0)} a(x,t).$$
 We distinguish two cases: $\bar{a}_{\rho}\leq 2[a]_{\alpha,\frac{\alpha}{2}}\rho^\alpha$ and $\bar{a}_{\rho}\geq 2[a]_{\alpha,\frac{\alpha}{2}}\rho^\alpha$. For the first case, we utilize the Caccioppoli type inequality \eqref{2.1} in $Q_j$ along with the property $\zeta(x,t)=1$ in $\tilde{Q}_j$, which gives that
\begin{align}
\label{4.3}
& \operatorname*{ess\sup}_{-\tilde{\theta}_j<t<0}\int_{\tilde{B}_j} w_j^2(x,t) \,dx+\int_{-\tilde{\theta}_j}^0 \int_{\tilde{B}_j} \int_{\tilde{B}_j} \frac{\left|w_j(x,t)-w_j(y,t)\right|^p}{|x-y|^{N+sp}}\,dxdydt\nonumber\\
\leq &C \int_{-\theta_j}^{0} \int_{B_j} w_j^2(x,t)|\partial_t \zeta^m|\,dxdt+ C \int_{-\theta_j}^{0} \int_{B_j}a(x,t)  w_j^q(x,t)|\nabla \zeta(x,t)|^q\,dxdt\nonumber\\
&+ C \int_{-\theta_j}^{0}\int_{B_j}\int_{B_j}(\max\{w_j(x,t), w_j(y,t)\})^p\left|\zeta^{\frac{m}{p}}(x,t)-\zeta^{\frac{m}{p}}(y,t)\right|^p \, d\mu dt\nonumber\\
&+C \mathop{\mathrm{ess}\,\sup}_{\stackrel{-\theta_j<t<0}{x \in \hat{B}_j}}  \int_{\mathbb{R}^N \backslash B_j} \frac{w_j^{p-1}(y,t)}{|x-y|^{N+sp}}\,dy \int_{-\theta_j}^{0} \int_{B_j} w_j(x,t) \zeta^m(x,t) \, dxdt,
\end{align}
where $m=\max\{p,q\}$. Denote $A_j:=[u<k_j]\cap Q_j$, we estimate
\begin{align}
\label{4.4}
\int_{-\theta_j}^{0} \int_{B_j} w_j^2(x,t)|\partial_t \zeta^m|\,dxdt\leq C \frac{2^{spj}}{\theta}L^2|A_j|.
\end{align}
By the condition \eqref{1.12}, it shows that
\begin{align}
\label{4.5}
\int_{-\theta_j}^{0} \int_{B_j}a(x,t)  w_j^q(x,t)|\nabla \zeta(x,t)|^q\,dxdt&\leq C \frac{2^{jq}}{\rho^q}\bar{a}_\rho L^q|A_j|\nonumber\\
&\leq C 2^{jq}\rho^{\alpha-q} M^{q-p}L^p|A_j|\nonumber\\
&\leq C\frac{2^{jq}}{\rho^{sp}}L^p|A_j|,
\end{align}
where $C$ depends on \textit{the data}, $[a]_{\alpha,\frac{\alpha}{2}}$ and $M$.
Moreover, we have
\begin{align}
\label{4.6}
\int_{-\theta_j}^{0}\int_{B_j}\int_{B_j}(\max\{w_j(x,t), w_j(y,t)\})^p\left|\zeta^{\frac{m}{p}}(x,t)-\zeta^{\frac{m}{p}}(y,t)\right|^p \, d\mu dt\leq C\frac{2^{jm}}{\rho^{sp}}L^p|A_j|.
\end{align}
For $|y|\geq \rho_j$ and $|x|\leq \hat{\rho}_j$, we can get
\begin{align*}
\frac{|y-x|}{|y|} \geq 1-\frac{\hat{\rho}_j}{\rho_j}=\frac{1}{4}\left(\frac{\rho_j-\rho_{j+1}}{\rho_j}\right) \geq \frac{1}{2^{j+4}},
\end{align*}
which along with \eqref{4.2} gives that
\begin{align}
\label{4.7}
\mathop{\mathrm{ess}\sup}_{\stackrel{-\theta_j<t<0}{x \in \hat{B}_j}}  \int_{\mathbb{R}^N \backslash B_j} \frac{w_j^{p-1}(y,t)}{|x-y|^{N+sp}}\,dy \int_{-\theta_j}^{0} \int_{B_j} w_j(x,t) \zeta^m(x,t) \, dxdt\leq C\frac{2^{j(N+sp)}}{\rho^{sp}}L^p|A_j|.
\end{align}

We deduce from \eqref{4.3}--\eqref{4.7} that
\begin{align*}
&\operatorname*{ess\sup}_{-\tilde{\theta}_j<t<0}\int_{\tilde{B}_j} w_j^2(x,t) \,dx+\int_{-\tilde{\theta}_j}^0 \int_{\tilde{B}_j} \int_{\tilde{B}_j} \frac{\left|w_j(x,t)-w_j(y,t)\right|^p}{|x-y|^{N+sp}}\,dxdydt\\
\leq &C 2^{j(N+2m)}\left(\frac{L^2}{\theta}+\frac{L^p}{\rho^{sp}}\right)|A_j|,
\end{align*}
where $C$ depends on \textit{the data}, $[a]_{\alpha,\frac{\alpha}{2}}$, and $M$.
Employing the Sobolev embedding Lemma \ref{lem-1-8} and recalling the fact $\zeta=1$ in $\tilde{Q}_j$ yields that
\begin{align}
\label{4.8}
&\iint_{\tilde{Q}_j}(w_j\zeta)^{p(1+\frac{2s}{N})}\,dxdt=\iint_{\tilde{Q}_j}w_j^{p(1+\frac{2s}{N})}\,dxdt\nonumber\\
\leq&\int_{-\tilde{\theta}_j}^0 \int_{\tilde{B}_j} \int_{\tilde{B}_j} \frac{\left|w_j(x,t)-w_j(y,t)\right|^p}{|x-y|^{N+sp}}\,dxdydt\times\left(\operatorname*{ess\sup}_{-\tilde{\theta}_j<t<0}\int_{\tilde{B}_j} w_j^2(x,t) \,dx\right)^{\frac{sp}{N}}\nonumber \\
\leq& C\left[2^{j(N+2m)}\left(\frac{L^2}{\theta}+\frac{L^p}{\rho^{sp}}\right)\right]^{\frac{N+sp}{N}}|A_j|^{\frac{N+sp}{N}}.
\end{align}
On the other hand,
\begin{align}
\label{4.9}
\iint_{\tilde{Q}_j}(w_j\zeta)^{p(1+\frac{2s}{N})}\,dxdt&\geq \iint_{\tilde{Q}_j}w_j^{p(1+\frac{2s}{N})}\chi_{[u<k_{j+1}]}\,dxdt\nonumber\\
&\geq \iint_{Q_{j+1}\cap[u<k_{j+1}]}(k_j-k_{j+1})^{p(1+\frac{2s}{N})}\,dxdt\nonumber\\
&\geq\left[\frac{(1-b)L}{2^{j+1}}\right]^{p\frac{N+2s}{N}}|Q_{j+1}\cap[u<k_{j+1}]|.
\end{align}
Set
\begin{align*}
Y_j=\frac{|[u<k_j]\cap Q_j|}{|Q_j|},
\end{align*}
then it follows from \eqref{4.8} and \eqref{4.9} that
\begin{align*}
Y_{j+1}\leq \frac{C h^j}{(1-b)^{p\frac{N+2s}{N}}}\left(\frac{\theta}{L^{2-p}}\right)^{\frac{sp}{N}}\left(\frac{1}{\rho^{sp}}+\frac{L^{2-p}}{\theta}\right)^{\frac{N+sp}{N}}\rho^{sp}|Y_j|^{\frac{N+sp}{N}},
\end{align*}
where $h=h(N,s,p,q)>1$. Since $\theta=\eta\rho^{sp}$, we can obtain $Y_j\rightarrow 0$ as $j\rightarrow\infty$ by Lemma \ref{lem-1-9} provided
\begin{align*}
Y_0\leq C^{-\frac{N}{sp}}h^{-\left(\frac{N}{sp}\right)^2}(1-b)^{N+2} \frac{L^{2-p}}{\eta}\left(1+\frac{L^{2-p}}{\eta}\right)^{-\frac{N+sp}{sp}}:=\nu.
\end{align*}
For the case $\bar{a}_{\rho}\geq 2[a]_{\alpha,\frac{\alpha}{2}}\rho^\alpha$, we can check that $\frac{2}{3}\bar{a}_\rho\leq a(x,t)\leq 2\bar{a}_\rho$ in $Q_{\rho,\eta\rho^q}(x_0,t_0)$ which leads to the nonstandard growth problem. In this case, we can use the method of proving Lemma 3.2 in our previous paper \cite{SZ23} to get this result. Similar to the proof above, employ Caccioppoli inequality \eqref{2.1} in $Q_j$ and Lemma 2.1 in \cite{SZ23}, we can arrive at the conclusion by choosing
\begin{align*}
\nu=C^{-\frac{N}{q}}\tilde{h}^{-\left(\frac{N}{q}\right)^2}(1-b)^{N+2} \frac{L^{2-q}}{\eta}\left(1+\frac{L^{2-q}}{\eta}+L^{p-q}\right)^{-\frac{N+q}{q}},
\end{align*}
where $\tilde{h}=\tilde{h}(N,s,p,q)>1$.
\end{proof}

We now prepare to prove Theorem \ref{thm-1-6}. From the definition of $u_*$, we can obtain
\begin{align}
\label{4.10}
u_*(x,t)=\operatorname*{ess\lim\inf}_{(y,\hat{t}) \rightarrow(x,t)}u(y,\hat{t})\leq\inf_{\eta>0} \lim _{\varrho \rightarrow 0}\operatorname*{ess\inf}_{Q'_{\rho,\theta}(x,t)} u \text{\ \ for every }(x,t)\in Q_T.
\end{align}
For the reverse inequality of \eqref{4.10}, we need prove a De Giorgi lemma over the forward cylinders with initial data as below.

\begin{lemma}
	\label{lem-4-3}
Let the assumptions \eqref{1.2} and \eqref{1.12} hold. Set $M:=\sup_{Q_T} |u|$ and fix the radius $\rho\in(0,1]$, the constants $L\in(0,M)$, $b\in(0,1]$. Assume that $u$ is a local weak supersolution to \eqref{1.1}. Set the same time levels as \eqref{4.1}.
Let the number $\mu^-$ satisfy
\begin{align*}
\mu^-\leq \operatorname*{ess\inf}_{\mathbb{R}^N\times(0,T)} u.
\end{align*}
There exists a constant $\eta>0$ depending on the data, $M,L,b,\mu^-$ and $[a]_{\alpha,\frac{\alpha}{2}}$ such that if
\begin{align*}
u(\cdot,t_0)\geq \mu^-+L \text{\ a.e. in  } B_\rho(x_0),
\end{align*}
then
\begin{align*}
u\geq \mu^-+bL \text{\ a.e. in }  Q_{\frac{\rho}{2},\frac{\theta}{2}}(x_0,t_0).
\end{align*}
\end{lemma}
\begin{proof}
We may assume $(x_0,t_0)=(0,0)$. Let $k_j,\rho_j,\tilde{\rho}_j,\hat{\rho}_j,\theta_j,\tilde{\theta}_j,\hat{\theta}_j,B_j,\tilde{B}_j,\hat{B}_j$ be as in Lemma \ref{lem-4-2}. Define the forward cylinders
\begin{align*}
Q_j=B_j\times(0,\theta_j),\ \tilde{Q}_j=\tilde{B}_j\times(0,\tilde{\theta_j}),\ \hat{Q}_j=\hat{B}_j\times(0,\hat{\theta}_j).
\end{align*}
Choose the piecewise cutoff function $\zeta(x,t)=\zeta(x)$ vanishing outside $\hat{B}_j$ such that
\begin{align*}
0\leq\zeta\leq 1, \quad \zeta(x)=1 \text{\ in } \tilde{B}_j, \quad |\nabla\zeta|\leq\frac{2^j}{\rho}.
\end{align*}
When $\bar{a}_{\rho}\leq 2[a]_{\alpha,\frac{\alpha}{2}}\rho^\alpha$, a combination of energy estimates \eqref{2.1} in the cylinder $\tilde{Q}_j$ and the fact $u(\cdot,0)\geq \mu^-+L$ yield that
\begin{align*}
&\operatorname*{ess\sup}_{0<t<\theta_j}\int_{\tilde{B}_j} w_j^2(x,t) \,dx+\int_0^{\theta_j} \int_{\tilde{B}_j} \int_{\tilde{B}_j} \frac{\left|w_j(x,t)-w_j(y,t)\right|^p}{|x-y|^{N+sp}}\,dxdydt\\
\leq &C 2^{j(N+2m)}\frac{L^p}{\rho^{sp}}|A_j|,
\end{align*}
where $A_j:=[u<k_j]\cap Q_j$. Let $Y_j=\frac{|A_j|}{|Q_j|}$. Following the same steps as proving Lemma \ref{lem-4-2}, we obtain
\begin{align*}
Y_{j+1}\leq \frac{C h^j}{(1-b)^{p\frac{N+2s}{N}}}\left(\frac{\eta}{L^{2-p}}\right)^{\frac{sp}{N}}|Y_j|^{\frac{N+sp}{N}}
\end{align*}
with $h=h(N,p,s,q)>1$. By Lemma \ref{lem-1-9}, it shows that if
\begin{align*}
Y_0\leq C^{-\frac{N}{sp}} h^{-\frac{N^2}{s^2p^2}}(1-b)^{N+2}\frac{L^{2-p}}{\eta},
\end{align*}
then $Y_j\rightarrow 0$. Thus, we can arrive at the claim by choosing
\begin{align*}
\eta=C^{-\frac{N}{sp}}(1-b)^{N+2s} L^{2-p}.
\end{align*}
Proceeding similarly as above, we may choose $$\eta=C^{-\frac{N}{q}}(1-b)^{N+2}L^{2-q}(1+L^{p-q})^{-\frac{N+q}{q}}$$ in the case $\bar{a}_{\rho}\geq 2[a]_{\alpha,\frac{\alpha}{2}}\rho^\alpha$ to finish the proof.
\end{proof}

The rest of the proof of the inverse inequality of \eqref{4.10} is based on Lemma \ref{lem-4-3}. Since it is similar to [Theorem 3.1, \cite{L21}], we omit the proof here. The proof of Theorem \ref{thm-1-6} is finally complete.

\section*{Acknowledgment}
This work was supported by the National Natural Science Foundation of China (No. 12071098).

\end{document}